\documentclass{amsart}[10pt,leqno]

\newtheorem{theorem}{Theorem}[section]
\newtheorem{lemma}[theorem]{Lemma}
\newtheorem{proposition}[theorem]{Proposition}

\newtheorem{problem}[theorem]{Problem}

\theoremstyle{definition}
\newtheorem{definition}[theorem]{Definition}
\newtheorem{example}[theorem]{Example}

\newtheorem{remark}[theorem]{Remark}

\newtheorem{algorithm}[theorem]{Algorithm}
\newtheorem{solution of system}[theorem]{Our solution}

\usepackage{amscd,amssymb}

\begin{document}

\title[]
{A Diophantine transport problem from 2016\\
    and its possible solution from 1903}
\author[]{Silvia Boumova, Vesselin Drensky, Boyan Kostadinov}
\date{}
\address{Faculty of Mathematics and Informatics,
Sofia University `` St. Kliment Ohridski'',
5 James Bourchier Blvd., 1164 Sofia, Bulgaria,
and Institute of Mathematics and Informatics,
Bulgarian Academy of Sciences,
Acad. G. Bonchev Str., Block 8,
1113 Sofia, Bulgaria}
\email{boumova@fmi.uni-sofia.bg, silvi@math.bas.bg}
\address{Institute of Mathematics and Informatics,
Bulgarian Academy of Sciences,
Acad. G. Bonchev Str., Block 8,
1113 Sofia, Bulgaria}
\email{drensky@math.bas.bg}
\address{Faculty of Applied Mathematics and Informatics,
Technical University of Sofia,
8 Kl. Ohridski Blvd., 1000 Sofia, Bulgaria}
\email{boyan.sv.kostadinov@gmail.com}

\thanks
{Partially supported by Grant KP-06-N-32/1 ``Groups and Rings -- Theory and Applications''
of the Bulgarian National Science Fund.}

\maketitle

\begin{abstract}
Motivated by a recent Diophantine transport problem about how to transport profitably a group of persons or objects,
we survey classical facts about solving systems of linear Diophantine equations and inequalities in nonnegative integers.
We emphasize on the method of Elliott from 1903 and its further development by MacMahon in his ``$\Omega$-Calculus'' or Partition Analysis.
As an illustration we obtain the solution of the considered transport problem in terms of
a formal power series in several variables which is an expansion of a rational function of a special form.
\end{abstract}

\section{Introduction}
The idea for this paper came from the very interesting recent papers by Robles-P\'{e}rez and Rosales \cite{RPR} and \cite{RPR1}.
Starting with a specific transport problem about how to transport profitably a group of persons or objects
the authors of \cite{RPR} and \cite{RPR1} have generalized it to
the following system of linear Diophantine inequalities.

Let $\mathbb{N}$ be the set of nonnegative integers, let $(a_1,\ldots,a_k)$ and $(b_1,\ldots,b_k)$ belong to $\mathbb{N}^k$,  and
let $a,b\in \mathbb{N}$. How to find the set $T$ of all solutions $y\in\mathbb{N}$ of the system
\[
\begin{array}{|ccl}
y&\ge & a_1 x_1 + \cdots + a_kx_k + a \\
y&\le & b_1 x_1 + \cdots + b_k x_k - b? \\
\end{array}
\]
The approach in \cite{RPR} and \cite{RPR1} is to prove that $T\cup\{0\}$ is a submonoid of the additive monoid $(\mathbb{N},+)$ and to develop an
algorithm for computing the minimal system of its generators. This is in the spirit of results in \cite{RGGB}, the main of which states
that there exists a one-to-one correspondence between the set of numerical semigroups (i.e. submonoids $S$ of $(\mathbb{N},+)$ such that
$\mathbb{N}\setminus S$ is a finite set) with a fixed minimal nonzero element and
the set of nonnegative integer solutions of a system of linear Diophantine inequalities.

Linear Diophantine equations and inequalities and their solutions in nonnegative integers
are classical objects which appear in many branches of mathematics, computer science and their applications.
Many methods have been developed for solving such systems. The purpose of our paper is to survey results, many of them with proofs,
starting from Euler, Gordan and Hilbert. Special attention is paid to
the method of Elliott \cite{E} from 1903 and its further development by MacMahon \cite{MM} in his ``$\Omega$-Calculus'' or Partition Analysis.
Later, this method was studied in detail by Stanley \cite{S1} who added new geometric ideas.
More recently, Domenjoud and Tom\'as \cite{DT} gave new life
to the method deriving an algorithm for solving systems of linear Diophantine equations, inequalities and
disequalities in nonnegative integers. As an illustration of the methods of Elliott and MacMahon
in the present paper we solve one of the Diophantine transport problems which motivated our project.

The ideas of Elliott and MacMahon have many other applications.
Andrews, alone or jointly with Paule, Riese, and Strehl published a series of twelve papers
 (I -- \cite{A}$,\ldots,$ XII -- \cite{AP}) on MacMahon's Partition Analysis, with numerous applications to different problems,
illustrating the power of the methods. The ``$\Omega$-Calculus'' was further improved by developing better algorithms and effective computer realizations
by Andrews, Paule, and Riese \cite{APR1, APR2}, and Xin \cite{X}. Other applications were given
by Berele \cite{B2, B3} (to algebras with polynomial identities), Bedratyuk and Xin \cite{BX} (to classical invariant theory),
and the authors in a series of papers, also jointly with Benanti, Genov, and Koev
(to algebras with polynomial identities, and classical and noncommutative invariant theory),
see \cite{BBDGK} and the references there.

Let
\[
a_{ij},a_i\in\mathbb{Z},\quad i=1,\ldots,m,j=1,\ldots,k,
\]
be arbitrary integers. Consider the system of Diophantine equations and inequalities
\begin{equation}\label{arbitrary transport system}
\begin{array}{|rcrcrl}
a_{11} x_1& + \cdots +& a_{1k} x_k &+& a_1&= 0 \\
&\cdots&&&&\\
a_{l1} x_1& + \cdots + &a_{lk} x_k &+& a_l&= 0 \\
a_{l+1,1} x_1& + \cdots +& a_{l+1,k} x_k&+& a_{l+1}&\geq 0 \\
&\cdots&&&&\\
a_{m1} x_1& + \cdots +& a_{mk} x_k&+&  a_m&\geq 0. \\
\end{array}
\end{equation}
Let the set of solutions of the system (\ref{arbitrary transport system}) in nonnegative integers be
\[
S=\{s=(s_1,\ldots,s_k)\in\mathbb{N}^k\mid s\text{ is a solution of the system}\}.
\]
The leitmotif of the paper is to apply a slight modification of the method as presented in the original paper by Elliott \cite{E},
to show how to calculate the function
\begin{equation}\label{characteristic transport series}
\chi_S(t_1,\ldots,t_k)=\sum_{s\in S}t_1^{s_1}\cdots t_k^{s_k},
\end{equation}
which describes the solutions of the system,
and to derive the parametric form of the solutions.
We give concrete calculations for the example in \cite{RPR}. Using similar methods one can handle also the example in \cite{RPR1}.

\section{The ideas of Euler, Gordan, and Hilbert seen from nowadays}
We start with some classical facts on systems of linear Diophantine equations.
For historical details and a survey of the methods for solving such systems we refer to
the sections on historical and further notes on linear Diophantine equations and on integer linear programming in the book of Schrijver \cite{Sch}
and the section of brief historical notes in the Ph.D. Thesis of Tom\'as \cite{T}.

The following fact was known already by Euler in 1748, see \cite{Eu1, Eu2, Eu3}.

\begin{lemma}\label{Lemma of Euler}
Let all coefficients $a_{ij}$ and $b_j$ of the system
\begin{equation}\label{system of Euler}
\begin{array}{|rcrl}
a_{11} x_1 &+ \cdots + &a_{1k} x_k =&b_1\\
&\cdots&&\\
a_{m1} x_1 &+ \cdots + &a_{mk} x_k =&b_m,\\
\end{array}
\end{equation}
be nonnegative integers such that for each $j=1,\ldots,k$ at least one of the coefficients
$a_{ij}$, $i=1,\ldots,m$, is different from $0$. Then the number of solutions in nonnegative integers of the system
is equal to the coefficient of $t_1^{b_1}\cdots t_m^{b_m}$ of the expansion as a power series of the product
\begin{equation}\label{coefficients of system of Euler}
\prod_{j=1}^k\frac{1}{1-t_1^{a_{1j}}\cdots t_m^{a_{mj}}}.
\end{equation}
\end{lemma}

\begin{proof}
Using the formula
\[
\frac{1}{1-z}=1+z+z^2+\cdots,
\]
we obtain immediately that
\[
\prod_{j=1}^k\frac{1}{1-t_1^{a_{1j}}\cdots t_m^{a_{mj}}}=\prod_{j=1}^k\sum_{s_j\geq 0}t_1^{a_{1j}s_j}\cdots t_m^{a_{mj}s_j}
\]
\[
=\sum_{s_j\geq 0}t_1^{\sum_{j=1}^ka_{1j}s_j}\cdots t_m^{\sum_{j=1}^ka_{mj}s_j}=\sum_{b_i\geq 0}c_bt_1^{b_1}\cdots t_m^{b_m},
\]
where the coefficient $c_b$ is equal to the number of $k$-tuples $(s_1,\ldots,s_k)\in{\mathbb N}^k$ such that
\[
a_{i1}s_1+\cdots+a_{ik}s_k=b_i,\quad i=1,\ldots,m,
\]
i.e. to the number of solutions of the system (\ref{system of Euler}).
\end{proof}

\begin{remark}\label{to find solutions}
If we replace in Lemma \ref{Lemma of Euler} the product (\ref{coefficients of system of Euler}) by
\[
\prod_{j=1}^k\frac{1}{1-z_jt_1^{a_{1j}}\cdots t_m^{a_{mj}}},
\]
then the coefficient of $t_1^{b_1}\cdots t_k^{b_k}$ will be a polynomial
\[
\chi_S(z_1,\ldots,z_k)=\sum_{s\in S}z_1^{s_1}\cdots z_k^{s_k}
\]
in $z_1,\ldots,z_k$, where $S\subset {\mathbb N}^k$ is the set of the solutions $s=(s_1,\ldots,s_k)$ of the system (\ref{system of Euler}).
\end{remark}

\begin{example}\label{example Euler}
Given the system
\[
\begin{array}{|rcrcrl}
x_1&+&x_2&+&x_3&=10\\
x_1&+&2x_2&+&3x_3&=15\\
\end{array}
\]
we expand the product
\[
\frac{1}{(1-z_1t_1t_2)(1-z_2t_1t_2^2)(1-z_3t_1t_2^3)}
\]
as a power series and find that the coefficient of $t_1^{10}t_2^{15}$ is equal to
$z_1^7z_2z_3^2+z_1^6z_2^3z_3+z_1^5z_2^5$. Hence the system has three solutions
\[
(s_1,s_2,s_3)=(7,1,2),(6,3,1),(5,5,0).
\]
\end{example}

Now we shall consider systems of homogeneous linear Diophantine equations
\begin{equation}\label{arbitrary homogeneous system}
\begin{array}{|rcrl}
a_{11} x_1& + \cdots +& a_{1k} x_k&= 0 \\
&\cdots&&\\
a_{m1} x_1& + \cdots +& a_{mk} x_k&= 0, \\
\end{array}
\end{equation}
where the coefficients $a_{ij}$ are arbitrary integers. As before, we shall be interested in solutions in nonnegative integers.
We introduce a partial order on ${\mathbb N}^k$:
\begin{equation}\label{partial order}
q'=(q_1',\ldots,q_k')\preceq (q_1'',\ldots,q_k'')=q''\text{ if }q_j'\leq q_j'',\quad j=1,\ldots,k.
\end{equation}
If $q'\prec q''$ are two solutions, then
\[
q=q''-q'=(q_1,\ldots,q_k)=(q_1''-q_1',\ldots,q_k''-q_k')\in{\mathbb N}^k
\]
is also a solution and $q''$ is a sum of two smaller solutions $q$ and $q'$. Hence every solution of the system (\ref{arbitrary homogeneous system})
is a sum of {\it minimal} (or {\it fundamental}) solutions. In 1873 Gordan \cite{G1} called the minimal solutions {\it irreducible}.
He proved that every system (\ref{arbitrary homogeneous system}) has a finite number of minimal solutions. Here we give the proof from the book
by Grace and Young \cite[Chapter VI, Section 97]{GY} which is very close to the original proof of Gordan.

\begin{theorem}\label{theorem of Gordan}
The system (\ref{arbitrary homogeneous system}) has a finite number of minimal solutions.
\end{theorem}

\begin{proof}
We start with a single equation. Changing the order of the unknowns we rewrite the equation in the form
\begin{equation}\label{single equation}
a_1x_1+\cdots+a_mx_m=b_1y_1+\cdots+b_ny_n,
\end{equation}
where all $a_i$ and $b_j$ are positive integers. The equation has $mn$ solutions
\[
x_r=b_s,y_s=a_r
\]
and all other variables equal to 0. Now, let us assume that in the solution
$(r,s)=(r_1,\ldots,r_m,s_1,\ldots,s_n)$ one of the coordinates $r_i$ (e.g. $r_1$) is greater than $b_1+\cdots+b_n$.
Hence
\[
b_1s_1+\cdots+b_ns_n=a_1r_1+\cdots+a_mr_m\geq a_1r_1>a_1(b_1+\cdots+b_m),
\]
\[
b_1(s_1-a_1)+\cdots+b_n(s_n-a_1)>0
\]
and there exists an $s_j$ (e.g. $s_1$) such that $s_j>a_1$. Then
\[
(r,s)=(r_1-b_1,r_2,\ldots,r_m,s_1-a_1,s_2,\ldots,s_n)+(b_1,0,\ldots,0,a_1,0,\ldots,0)
\]
and the solution $(r,s)$ is not minimal. In this way the minimal solutions satisfy the condition
\[
r_i\leq b_1+\cdots+b_n,s_j\leq a_1+\cdots+a_m.
\]
Hence they are a finite number and can be found explicitly.
Let $(r^{(1)},s^{(1)}),\ldots,(r^{(p)},s^{(p)})$ be all minimal solutions of the first equation of the system (\ref{arbitrary homogeneous system}).
Then all solutions of the first equation are in the form
\begin{equation}\label{minimal solutions of one equation}
q=(r,s)=(r^{(1)},s^{(1)})z_1+\cdots+(r^{(p)},s^{(p)})z_p,\quad z_i\in\mathbb{N}.
\end{equation}
Replacing $(r,s)$ in the second equation of the system (\ref{arbitrary homogeneous system}) we obtain an equation
\begin{equation}\label{second equation}
c_1z_1+\cdots+c_pz_p=0
\end{equation}
with unknowns $z_1,\ldots,z_p$.
Then the minimal solutions $q=(q_1,\ldots,q_k)$ of the first two equations of (\ref{arbitrary homogeneous system})
are among the solutions (\ref{minimal solutions of one equation})
obtained from the minimal solutions of the equation (\ref{second equation}).
Again, we can find them explicitly. Continuing in the same way we can find all minimal solutions of the system
(\ref{arbitrary homogeneous system}).
\end{proof}

\begin{remark}\label{solution by method of Euler}
We can find the candidates for the minimal solutions of the equation (\ref{single equation})
combining the method of the proof of Theorem \ref{theorem of Gordan} with the method of Euler from Lemma \ref{Lemma of Euler}
as modified in Remark \ref{to find solutions}. We consider the power series
\[
T(t_1,\ldots,t_m,z)=\prod_{i=1}^m\frac{1}{1-t_iz^{a_i}}=\sum_{x_i\geq 0} t_1^{x_1}\cdots t_m^{x_m}z^{a_1x_1+\cdots+a_mx_m}=\sum_{k\geq 0}T_k(t_1,\ldots,t_m)z^k,
\]
\[
U(u_1,\ldots,u_m,z)=\prod_{j=1}^n\frac{1}{1-u_jz^{b_j}}=\sum_{y_j\geq 0} u_1^{y_1}\cdots u_n^{y_n}z^{b_1y_1+\cdots+b_ny_n}=\sum_{k\geq 0}U_k(u_1,\ldots,u_n)z^k.
\]
The polynomials $T_k(t_1,\ldots,t_m)$ and $U_k(u_1,\ldots,u_n)$, respectively, are
sums of monomials of the form $t_1^{x_1}\cdots t_m^{x_m}$ and $u_1^{y_1}\cdots u_n^{y_n}$
and each pair of these monomials gives a solution of (\ref{single equation}) of the form
\[
a_1x_1+\cdots+a_mx_m=b_1y_1+\cdots+b_ny_n=k.
\]
By the proof of Theorem \ref{theorem of Gordan} the candidates for minimal solutions satisfy the conditions $x_i\leq b_1+\cdots+b_n$  and $y_j\leq a_1+\cdots+a_m$.
Hence it is sufficient to compute the polynomials $T_k(t_1,\ldots,t_m)$ and $U_k(u_1,\ldots,u_n)$ for $k\leq (a_1+\cdots+a_m)(b_1+\cdots+b_n)$.
\end{remark}

\begin{example}\label{solution of system by Gordan}
Let us consider the system
\begin{equation}\label{example homogeneous system}
\begin{array}{|rcrcrrl}
x_1&+&2x_2&-&x_3&-x_4&=0\\
2x_1&+&3x_2&-&2x_3&-x_4&=0.\\
\end{array}
\end{equation}
We rewrite the first equation in the form
\[
x_1+2x_2=y_3+y_4.
\]
By the proof of Theorem \ref{theorem of Gordan} every minimal solution $q=(r_1,r_2,s_1,s_2)$ satisfies the conditions
\[
0<r_1+r_2,\quad r_1,r_2\leq 2,\quad 0<s_1+s_2,\quad s_1,s_2\leq 3.
\]
There are $8$ possibilities for $r_1+r_2$ and for each $(r_1,r_2)$
there are $r=r_1+r_2+1$ possibilities $(r,0),(r-1,1),\ldots,(0,r)$ for $(s_1,s_2)$.
Simple calculations give that there are $35$ candidates for minimal solutions.
We start with the cases $(r_1,r_2)=(0,1)$ and $(1,0)$ and obtain $5$ solutions
\begin{equation}\label{fisrt minimal solutions}
\begin{array}{c}
q^{(1)}=(0,1,0,2),\quad q^{(2)}=(0,1,1,1),\quad q^{(3)}=(0,1,2,0),\\
\\
q^{(4)}=(1,0,0,1),\quad q^{(5)}=(1,0,1,0).\\
\end{array}
\end{equation}
If $r_1>1$, then $s_1+s_2>1$ and the solution $q=(r_1,r_2,s_1,s_2)$ is not minimal because $q^{(i)}\prec q$ for some $i=4,5$.
By similar argument we derive that the other solutions with $r_1+r_2>1$ are also not minimal. Thus we obtain that the minimal solutions of
the first equation of (\ref{example homogeneous system}) are those in (\ref{fisrt minimal solutions}) and all solutions of this equation are
\[
q=\sum_{i=1}^5t_iq^{(i)}=(t_4+t_5,t_1+t_2+t_3,t_2+2t_3+t_5,2t_1+t_2+t_4),\quad t_i\geq 0.
\]
The second equation of (\ref{example homogeneous system}) becomes
\[
2x_1+3x_2-2x_3-x_4=2(t_4+t_5)+3(t_1+t_2+t_3)-2(t_2+2t_3+t_5)-(2t_1+t_2+t_4)
\]
\[
=t_1-t_3+t_4=0.
\]
By the proof of Theorem \ref{theorem of Gordan} again, $t_1,t_4\leq 1$, $t_3\leq 2$,
and the candidates for minimal solutions of the equation $t_1-t_3+t_4=0$ are
\[
(t_1,t_3,t_4)=(1,1,0),(0,1,1),(1,2,1).
\]
Only the first two solutions are minimal. We have to add also the minimal solutions
\[
(t_1,t_2,t_3,t_4,t_5)=(0,1,0,0,0),(0,0,0,0,1)
\]
and obtain the candidates for minimal solutions of the original system (\ref{example homogeneous system})
\[
q=(0,2,2,2),(1,1,2,1),(0,1,1,1),(1,0,1,0).
\]
The first two solutions are not minimal and we obtain all minimal solutions of (\ref{example homogeneous system})
\[
q^{(1)}=(0,1,1,1),q^{(2)}=(1,0,1,0).
\]
\end{example}

\begin{example}\label{the same example}
We consider the same system (\ref{example homogeneous system}).
Applying the method in Remark \ref{solution by method of Euler} to the first equation $x_1+2x_2=y_3+y_4$ of the system,
we start with the power series
\[
T(t_1,t_2,z)=\frac{1}{(1-t_1z)(1-t_2z^2)}=\sum_{k\geq 0}T_k(t_1,t_2)z^k,
\]
\[
U(u_1,u_2,z)=\frac{1}{(1-u_1z)(1-u_2z)}=\sum_{k\geq 0}U_k(u_1,u_2)z^k
\]
and compute the first $6$ polynomials $T_k$ and $U_k$ ($k=1,\ldots,6$). For example
\[
T_1=t_1, U_1=u_1+u_2,\quad T_2=t_1^2+t_2, U_2=u_1^2+u_1u_2+u_2^2,
\]
which gives the pair of monomials
\[
(t_1,u_1),(t_1,u_2), (t_1^2,u_1^2),(t_1^2,u_1u_2),(t_1^2,u_2^2), (t_2,u_1^2),(t_2,u_1u_2),(t_2,u_2^2)
\]
producing the solutions
\[
(1,0,1,0),(1,0,0,1),(2,0,2,0),(2,0,1,1),(2,0,0,2),(0,1,2,0),(0,1,1,1),(0,1,0,2).
\]
Then by direct verification we select the minimal solutions of the equation and continue with the second equation of the system
(\ref{example homogeneous system}).
\end{example}

By the Hilbert Basis theorem \cite{H} every ideal of the polynomial algebra ${\mathbb Q}[x_1,\ldots,x_k]$ is finitely generated.
The following property of the partial order (\ref{partial order}) is known as the {\it Dickson lemma} with easy proof by induction in \cite{Di}.
As it is mentioned in \cite{Di} it is a direct consequence of the Hilbert Basis theorem applied to monomial ideals in ${\mathbb Q}[x_1,\ldots,x_k]$.

\begin{lemma}\label{proof of Hilbert}
Let $J$ be a subset of ${\mathbb N}^k$. Then $J$ has a finite subset
\[
\{q^{(i)}=(q_1^{(i)},\ldots,q_k^{(i)})\mid i=1,\ldots,n\}
\]
with the property that for any $q=(q_1,\ldots,q_k)\in J$ there exists a $q^{(i)}$ such that $q^{(i)}\preceq q$.
\end{lemma}

We should mention that this lemma was used by Gordon \cite{G2} in 1899 in his proof of the Hilbert Basis theorem.
Clearly, as a corollary we immediately obtain also a nonconstructive proof of Theorem \ref{theorem of Gordan}.

It is interesting to know the behavior of the minimal solutions of the system (\ref{arbitrary homogeneous system}).
It can be given in terms of recursion theory.

\begin{definition}\label{primitive recursive function in f}
Let $f:{\mathbb N}\to {\mathbb N}$ be an arbitrary function.
A function $h:{\mathbb N}^k\to {\mathbb N}$ is {\it primitive recursive in $f$} if $h$ can be obtained by a finite number of
steps applying the following rules, starting with the function $f$, the constant function $0$, the successor function $s:{\mathbb N}\to {\mathbb N}$
(defined by $s(n)=n+1$, $n\in\mathbb N$) and the projection function $p^k_i:{\mathbb N}^k\to {\mathbb N}$, $i=1,\ldots,k$ (defined by
$p^k_i(n_1,\ldots,n_k)=n_i$, $(n_1,\ldots,n_k)\in {\mathbb N}^k$):

(1) The functions $f$, $0$, $s$ and $p^k_i$ are primitive recursive in $f$;

(2) {\it Substitution:} If $g:{\mathbb N}^k\to {\mathbb N}$ and $h_i:{\mathbb N}^m\to {\mathbb N}$, $m=1,\ldots,k$, are primitive recursive in $f$, then
the function $g(h_1,\ldots,h_k):{\mathbb N}^m\to {\mathbb N}$ is also primitive recursive in $f$;

(3) {\it Primitive recursion:} If $g:{\mathbb N}^k\to {\mathbb N}$ and $h:{\mathbb N}^{k+2}\to {\mathbb N}$, are primitive recursive in $f$, then
the primitive recursion $p:{\mathbb N}^{k+1}\to {\mathbb N}$ of $g$ and $h$ defined by
\[
p(0,n_1,\ldots,n_k)=g(n_1,\ldots,n_k)\text{ and }p(s(m),n_1,\ldots,n_k)=h(m,p(m,n_1,\ldots,n_k)),
\]
$m\in{\mathbb N}$, $(n_1,\ldots,n_k)\in{\mathbb N}^k$,
is also primitive recursive in $f$.
\end{definition}

In the above definition, the ``ordinary'' primitive recursive functions are those which do not depend on the function $f$.
Roughly speaking, from the point of view of computability theory, a primitive recursive function can be computed by a computer program
such that for every loop in the program the number of iterations can be bounded from above before entering the loop.

\begin{definition}\label{recursive function in f}
The function $g:{\mathbb N}^k\to \mathbb N$ is {\it recursive} if in addition to the constructions in the definition of a primitive recursive function
one uses also the following.

(4) {\it Minimization operator} $\mu$: If $h(m,n_1,\ldots,n_k):{\mathbb N}^{k+1}\to \mathbb N$ is partially defined
(i.e. defined on a subset of ${\mathbb N}^{k+1}$),
then the function $\mu(h):{\mathbb N}^k\to \mathbb N$ is defined by $\mu(h)(n_1,\ldots,n_k)=m$ if $h(i,n_1,\ldots,n_k)>0$ for $i=0,1,\ldots,m-1$ and
$h(m,n_1,\ldots,n_k)=0$. If $h(i,n_1,\ldots,n_k)>0$ for all $i\in\mathbb N$ or if $h(i,n_1,\ldots,n_k)$
is not defined before reaching some $m$ with $h(m,n_1,\ldots,n_k)=0$,
then the search for $m$ never terminates, and $\mu(h)(n_1,\ldots,n_k)$ is not defined for the argument $(n_1,\ldots,n_k)$.
\end{definition}

We shall restate the formalization of Seidenberg \cite{Sa} introduced originally for the ideals of the polynomial algebra ${\mathbb Q}[x_1,\ldots,x_k]$.

\begin{problem}\label{approach of Seidenberg}
Given a function $f: {\mathbb N}\to{\mathbb N}$, what is the maximal $p_f(i)\in\mathbb N$ with the property:
There exists a set
\[
I_f=\{q^{(i)}=(q^{(i)}_1,\ldots,q^{(i)}_k)\mid i=1,\ldots,p_f(i)\}\subset {\mathbb N}^k
\]
such that $q^{(i)}_1+\cdots+q^{(i)}_k\leq f(i)$ and the elements of $I_f$ are not comparable with respect to the partial order $\prec$.
\end{problem}

Seidenberg showed that there exists a bound $p_f^{(k)}$ depending on $f$ and $k$ only, which is recursive in $f$ for a fixed $k$.
This result was improved by Moreno-Soc\'{\i}as \cite{MS2}.

\begin{theorem}\label{result of Moreno-Socias}
In the notation of Problem \ref{approach of Seidenberg}
for every $k$ there is a primitive recursive function $p_f^{(k)}:{\mathbb N}\to\mathbb N$ in $f$,
but there is no bound $p_f$ which is primitive recursive in $f$ in general.
\end{theorem}

For each $d\in{\mathbb N}\setminus\{0\}$ Moreno-Soc\'{\i}as \cite{MS1}
constructed an example for the primitive recursive function $f_d(n)=d+n$, $n\in\mathbb N$,
with the property that
the bound $p_{f_d}$ is expressed in terms of the Ackermann function $a(k,n):{\mathbb N}^2\to\mathbb N$ \cite{Ack} defined by
\[
a(0,n)=n+1,\quad a(k+1,0)=a(k,1),\quad a(k+1,n+1)=a(k,a(k+1,n)).
\]
It is known that $a(k,n)$ is recursive and grows faster than any primitive recursive function.

\begin{theorem}\label{Ackermann function}
In the notation of Problem \ref{approach of Seidenberg} let $d\in {\mathbb N}\setminus \{0\}$ and let $f_d(n)=d+n$, $n\in\mathbb N$.
Then there exists a set
\[
I_{f_d}=\{q^{(i)}=(q^{(i)}_1,\ldots,q^{(i)}_k)\mid i=d,d+1,\ldots,p\}\subset {\mathbb N}^k
\]
of noncomparable $k$-tuples such that $q^{(i)}_1+\cdots+q^{(i)}_k=i$ and $p$ is equal to
$a(k,d-1)-1$, where $a(k,n)$ is the Ackermann function.
\end{theorem}

See also the paper by Aschenbrenner and Pong \cite{APo} for another approach to the complexity of the problems discussed above.

\section{From the homogeneous case to the solution\\
of general linear Diophantine constraints}

Let us consider the general system of linear Diophantine equations and inequalities (\ref{arbitrary transport system}).
There is a standard way to bring the solution of (\ref{arbitrary transport system})
to the solution of a homogeneous system of linear Diophantine equations. We introduce new unknowns $y,y_{l+1},\ldots,y_m$
and replace the system (\ref{arbitrary transport system}) with the system
\begin{equation}\label{from arbitrary to homogeneous system}
\begin{array}{|rcrcrcll}
a_{11} x_1& + \cdots +& a_{1k} x_k &+& a_1y&&&= 0 \\
&\cdots&&&&&&\\
a_{l1} x_1& + \cdots + &a_{lk} x_k &+& a_ly&&&= 0 \\
a_{l+1,1} x_1& + \cdots +& a_{l+1,k} x_k&+& a_{l+1}y&-&y_{l+1}&=0 \\
&\cdots&&&&&&\\
a_{m1} x_1& + \cdots +& a_{mk} x_k&+&  a_my&-&y_m&=0. \\
\end{array}
\end{equation}

The following easy theorem shows how to reduce the solution of a general system to a homogeneous one.

\begin{theorem}\label{reduction from arbitrary to homogeneous}
Let
\begin{equation}\label{minimal solutions of the new system}
\{q^{(i)}=(r_1^{(i)},\ldots,r_k^{(i)},s^{(i)},s_{l+1}^{(i)},\ldots,s_m^{(i)})\mid i=1,\ldots,n\}\subset {\mathbb N}^{k+1+m-l}
\end{equation}
be the set of minimal solutions of the homogeneous system of equations (\ref{from arbitrary to homogeneous system}) and let
$s^{(i)}=1$ for $i=1,\ldots,c$, $s^{(i)}=0$ for $i=c+1,\ldots,d$ and $s^{(i)}>1$ for $i=d+1,\ldots,n$. Then the set of all
solutions of the system (\ref{arbitrary transport system}) are of the form
\[
q=q^{(i)}+t_{c+1}q^{(i)}+\cdots+t_dq^{(d)},\quad i=1,\ldots,c,\quad t_{c+1},\ldots,t_d\in{\mathbb N}.
\]
\end{theorem}

\begin{proof}
If $q=(r_1,\ldots,r_k,s,s_{l+1},\ldots,s_m)$ is a solution of (\ref{from arbitrary to homogeneous system}), then
$s_{l+1},\ldots,s_m\geq 0$ and the solutions $(r_1,\ldots,r_k)$ of (\ref{arbitrary transport system}) are obtained from
solutions $q$ with $s=1$. Every $q$ is a linear combination of the minimal solutions from (\ref{minimal solutions of the new system}).
If $q$ has the form $\displaystyle q=\sum_{i=1}^nt_iq^{(i)}$,
we obtain that $1=s=t_1+\cdots+t_c+t_{d+1}s^{(d+1)}+\cdots+t_ns^{(n)}$. Hence
$t_1+\cdots+t_c=1$ and $t_{d+1}=\cdots=t_n=0$ because $t_{d+1},\ldots,t_n>1$.
\end{proof}

\begin{remark}
Let some of the inequalities in the system (\ref{arbitrary transport system}), e.g.
\[
a_{m1} x_1 + \cdots + a_{mk} x_k+  a_m>0
\]
be strict. Then in the system (\ref{from arbitrary to homogeneous system}) we replace it by
\[
a_{m1} x_1 + \cdots + a_{mk} x_k+  (a_m+1)y-y_m=0.
\]
\end{remark}

\begin{example}\label{solution of general system}
Let us modify the system (\ref{example homogeneous system}) from Example \ref{solution of system by Gordan} into the system
\begin{equation}\label{example general system}
\begin{array}{|rcrrrl}
x_1&+&2x_2&-x_3&-1&=0\\
2x_1&+&3x_2&-2x_3&-1&\geq 0.\\
\end{array}
\end{equation}
Following the proof of Theorem \ref{reduction from arbitrary to homogeneous} we have to consider the system
\begin{equation}\label{example from general to homogeneous}
\begin{array}{|rcrrrrl}
x_1&+&2x_2&-x_3&-y&&=0\\
2x_1&+&3x_2&-2x_3&-y&-y_2&=0.\\
\end{array}
\end{equation}
The minimal solutions $q=(r_1,r_2,r_3,s)$ of the first equation of (\ref{example from general to homogeneous}) are the same
as the minimal solutions (\ref{fisrt minimal solutions}) of the first equation of (\ref{example homogeneous system}).
Since $s=2>1$ in the solution $q^{(1)}=(0,1,0,2)$, $s=1$ in $q^{(2)}$ and $q^{(4)}$, and $s=0$ in $q^{(3)}$ and $q^{(5)}$,
we obtain that all solutions are
\[
q'=q^{(2)}+t_3q^{(3)}+t_5q^{(5)}=(t_5,1+t_3,1+2t_3+t_5,1),
\]
\[
q''=q^{(4)}+t_3q^{(3)}+t_5q^{(5)}=(1+t_5,t_3,2t_3+t_5,1),
\]
$t_3,t_5\in{\mathbb N}$. Substituting $q'$ in the second equation of (\ref{example from general to homogeneous}) we obtain
\[
2t_5+3(1+t_3)-2(1+2t_3+t_5)-1-y_2=-t_3-y_2=0.
\]
Hence $t_3=y_2=0$ and we obtain the solutions $(r_1,r_2,r_3)=(1+t_5,0,t_5)$ of (\ref{example general system}).
Similarly, starting with $q''$, the second equation of (\ref{example from general to homogeneous}) gives
\[
2(1+t_5)+3t_3-2(2t_3+t_5)-1-y_2=1-t_3-y_2=0.
\]
We obtain two solutions $(t_3,y_2)=(0,1)$ and $(t_3,y_2)=(1,0)$, which give the solutions
$(r_1,r_2,r_3)=(1+t_5,0,t_5)$ and $(r_1,r_2,r_3)=(1+t_5,1,2+t_5)$ of (\ref{example general system}).
\end{example}

There are many methods based on different ideas for solving systems of linear Diophantine equations and inequalities.
See for example \cite{AbM, AC, BKo, C, DT, FT, Kr, Pa, PCVT, TF1, TF2} and the bibliography there.
See also \cite{Ko1, Ko2} for relations with mathematical logic and the theory of formal grammars.

\section{The method of Elliott}

In this section we shall explain in detail the method of Elliott from \cite{E}.
Originally, it was developed for systems of homogeneous linear Diophantine equations.
But for our applications we shall restate the method for systems of linear Diophantine inequalities.
We fix arbitrary integers
\[
a_{ij},a_i\in\mathbb{Z},\quad i=1,\ldots,m,j=1,\ldots,k,
\]
a system of Diophantine inequalities
\begin{equation}\label{arbitrary system}
\begin{array}{|cl}
a_{11} x_1 + \cdots + a_{1k} x_k + a_1\geq 0 \\
\cdots&\\
a_{m1} x_1 + \cdots + a_{mk} x_k + a_m\geq 0, \\
\end{array}
\end{equation}
and consider the set of solutions of the system (\ref{arbitrary system}) in nonnegative integers
\[
S=\{s=(p_1,\ldots,p_k)\in\mathbb{N}^k\mid s\text{ is a solution of the system}\}.
\]

\subsection{The first idea of Elliott.}
A usual way to describe a set $A\subset \mathbb{R}^k$
is in terms of its {\it characteristic} (or {\it indicator}) {\it function}
$\text{ch}_{f_A}:\mathbb{R}^k\rightarrow \{0,1\} $ defined by
\[
\text{ch}_{f_A}(p_1,\ldots,p_k) =
\begin{cases}
1,(p_1,\ldots,p_k) \in A\\
0,(p_1,\ldots,p_k) \notin A.\\
\end{cases}
\]
\begin{definition}
Let $P\subset \mathbb{N}^k$. By analogy with the characteristic function of $P$ we call the formal power series
\[
\chi_P(t_1,\ldots,t_k)=\sum_{p\in P}t_1^{p_1}\cdots t_k^{p_k},\quad p=(p_1,\ldots,p_k),
\]
the {\it characteristic series} of $P$.
\end{definition}
When $P\subset \mathbb{N}^k$ is the set of solutions of a system of linear Diophantine equations Elliott suggests to
call $\chi_P(t_1,\ldots,t_k)$ the {\it generating function} of the set of solutions.

\subsection{The second idea of Elliott.}
To find the characteristic series of a set of solutions $S\subset \mathbb{N}^k$ of a system of homogeneous linear Diophantine equations
Elliott involves Laurent series.

\begin{definition}\label{characteristic series of function}
Let $P\subset \mathbb{N}^k$, let
\[
f(t_1,\ldots,t_k)=\sum_{p\in P}\alpha_pt_1^{p_1}\cdots t_k^{p_k}\in\mathbb{C}[[t_1,\ldots,t_k]],\quad p=(p_1,\ldots,p_k),\alpha_p\in\mathbb{C},
\]
be a formal power series, and let $S$ be a subset of $P$. We call the formal power series
\[
\chi_S(f;t_1,\ldots,t_k)=\sum_{p\in S}\alpha_pt_1^{p_1}\cdots t_k^{p_k},
\]
the {\it characteristic series of $f$ with respect to the set $S$}.
\end{definition}

The next lemma is one of the key moments in the approach of Elliott. Its proof is obvious.

\begin{lemma}\label{Laurent series for equations}
Let $P\subset \mathbb{N}^k$, let
\[
f(t_1,\ldots,t_k)=\sum_{p\in P}\alpha_pt_1^{p_1}\cdots t_k^{p_k}\in\mathbb{C}[[t_1,\ldots,t_k]],\quad p=(p_1,\ldots,p_k),\alpha_p\in\mathbb{C},
\]
be a formal power series, and let $S\subset P$ be the set of solutions in $P$ of the Diophantine equation
\[
a_1x_1+\cdots+a_kx_k=0,\quad a_i\in\mathbb{Z}.
\]
If the Laurent series
\[
\xi_S(t_1,\ldots,t_k,z)=f(t_1z^{a_1},\ldots,t_kz^{a_k})=\sum_{n=-\infty}^{\infty}\sum_{p\in S}\alpha_pt_1^{p_1}\cdots t_k^{p_k}z^n,
\]
$p=(p_1,\ldots,p_k)$, $n=a_1p_1+\cdots+a_kp_k$, has the form
\[
\xi_S(t_1,\ldots,t_k,z)=\sum_{n=-\infty}^{\infty}f_n(t_1,\ldots,t_k)z^n,\quad
f_n(t_1,\ldots,t_k)\in \mathbb{C}[[t_1,\ldots,t_k]],
\]
then
\[
\chi_S(f;t_1,\ldots,t_k)=\sum_{p\in S}\alpha_pt_1^{p_1}\cdots t_k^{p_k}=f_0(t_1,\ldots,t_k).
\]
\end{lemma}

The next lemma is a slight generalization of the original approach of Elliott. Its proof is also obvious.

\begin{lemma}\label{Laurent series}
Let $P\subset \mathbb{N}^k$, let
\[
f(t_1,\ldots,t_k)=\sum_{p\in P}\alpha_pt_1^{p_1}\cdots t_k^{p_k}\in\mathbb{C}[[t_1,\ldots,t_k]],\quad p=(p_1,\ldots,p_k),\alpha_p\in\mathbb{C},
\]
be a formal power series, and let $S$ be  the solutions in $P$ of the Diophantine inequality
\[
a_1x_1+\cdots+a_kx_k+a\geq 0,\quad a_i,a\in\mathbb{Z}.
\]
If
\[
\xi_S(t_1,\ldots,t_k,z)=z^af(t_1z^{a_1},\ldots,t_kz^{a_k})
\]
\[
=\sum_{n=-\infty}^{\infty}\sum_{p\in S}\alpha_pt_1^{p_1}\cdots t_k^{p_k}z^n
=\sum_{n=-\infty}^{\infty}f_n(t_1,\ldots,t_k)z^n,
\]
$f_n(t_1,\ldots,t_k)\in \mathbb{C}[[t_1,\ldots,t_k]]$, $n=a_1p_1+\cdots+a_kp_k+a$, then
\[
\chi_{S}(f;t_1,\ldots,t_k)=\sum_{(p_1,\ldots,p_k)\in S}\alpha_pt_1^{p_1}\cdots t_k^{p_k}=\sum_{n=0}^{\infty}f_n(t_1,\ldots,t_k).
\]
\end{lemma}

Now the problem is how to find $\chi_{S}(f;t_1,\ldots,t_k)$
if $f(t_1,\ldots,t_k)$ is an explicitly given power series which converges to a rational function and we know the set $S$.
Even in simple cases the answer may be not trivial.

\begin{example}\label{invariant theory for free algebras}
Combining the results from \cite[p. 409]{DG} with ideas from \cite{BBDGK}, for the formal power series
\[
f(t_1,t_2)=\frac{t_1-t_2}{1-(t_1+t_2)}=(t_1-t_2)\sum_{n=0}^{\infty}(t_1+t_2)^n,\quad S=\{(p_1,p_2)\in\mathbb{N}^2\mid p_1\geq p_2\},
\]
we obtain
\[
\chi_S(f;t_1,t_2)=\frac{1-\sqrt{1-4t_1t_2}}{2t_2-(1-\sqrt{1-4t_1t_2})}.
\]
\end{example}

\subsection{The third idea of Elliott.}

The following definition is given by Berele \cite{B2}.

\begin{definition}\label{definition of nice rational function}
A {\it nice rational function} is a rational function
with denominator which is a product of monomials of the form
$(1-t_1^{\alpha_1}\cdots t_k^{\alpha_k})$.
\end{definition}

Nice rational functions appear in many places of mathematics.
For example, the Hilbert series of any finitely generated multigraded commutative algebra is of this form.
The following theorem was proved by Elliott \cite{E}.
The proof also gives an algorithm how to find the characteristic series $\chi_S(t_1,\ldots,t_k)$ of the set $S$ of solutions.

\begin{theorem}\label{theorem of Elliott}
Let
\begin{equation}\label{system of homogeneous Diophantine equations}
\begin{array}{|cl}
a_{11} x_1 + \cdots + a_{1k} x_k=0 \\
\cdots&\\
a_{m1} x_1 + \cdots + a_{mk} x_k=0, \\
\end{array}
\end{equation}
where $a_{ij}\in\mathbb{Z}$, $i=1,\ldots,m$, $j=1,\ldots,k$,
be a system of homogenous linear Diophantine equations. Then
the characteristic series $\chi_S(t_1,\ldots,t_k)$ of the set $S$ of the solutions of (\ref{system of homogeneous Diophantine equations})
in $\mathbb{N}^k$ is a nice rational function.
\end{theorem}

\begin{proof}
We start with the characteristic series of the set $\mathbb{N}^k$ of all points with nonnegative integers coordinates
\[
\chi_{\mathbb{N}^k}(t_1,\ldots,t_k)=\prod_{i=1}^k\frac{1}{1-t_i}=\sum_{p_i\geq 0}t_1^{p_1}\cdots t_k^{p_k}.
\]
Let the first equation of the system be of the form
\begin{equation}\label{one single equation}
a_1x_1+\cdots+a_dx_d-c_{e+1}x_{e+1}-\cdots-c_kx_k=0,\quad a_i> 0, c_j>0, d\leq e.
\end{equation}
Then the Laurent series
\[
\xi_S(t_1,\ldots,t_k,z)=\chi_{\mathbb{N}^k}(t_1z^{a_1},\ldots,t_dz^{a_d},t_{d+1},\ldots,t_e,t_{e+1}z^{-c_{e+1}},\ldots,t_kz^{-c_k})
\]
from Lemma \ref{Laurent series for equations} has the form
\[
\xi_S(t_1,\ldots,t_k,z)=\prod_{i=1}^d\prod_{j=d+1}^e\prod_{m=e+1}^k\frac{1}{(1-t_iz^{a_i})(1-t_j)(1-t_mz^{-c_m})}.
\]
More general, we shall assume that $\xi_S(t_1,\ldots,t_k,z)$ is of the form
\begin{equation}\label{products with positive and negative degrees}
\xi_S(t_1,\ldots,t_k,z)=\prod_{i=1}^d\prod_{j=d+1}^e\prod_{m=e+1}^k\frac{1}{(1-A_iz^{a_i})(1-B_j)(1-C_mz^{-c_m})},
\end{equation}
where $A_i,B_j,C_m$ are monomials in $t_1,\ldots,t_k$.

{\bf Case 1.} If the equation (\ref{one single equation}) does not contain negative $c_m$, i.e. if $e=k$, then obviously
\[
\xi_S(t_1,\ldots,t_k,z)=\prod_{i=1}^d\prod_{j=d+1}^k\frac{1}{(1-A_iz^{a_i})(1-B_j)}
=\prod_{j=d+1}^k\frac{1}{1-B_j}\left(1+\sum_{n\geq 1}D_nz^n\right),
\]
$D_n\in{\mathbb Q}[[t_1,\ldots,t_k]]$,
and
\[
\chi_S(t_1,\ldots,t_k)=f_0(t_1,\ldots,t_k)=\prod_{j=d+1}^k\frac{1}{1-B_j}.
\]

{\bf Case 2.}
Similar arguments work when the equation (\ref{one single equation}) does not contain positive $a_i$, i.e. when $d=0$.
Again
\[
\chi_S(t_1,\ldots,t_k)=f_0(t_1,\ldots,t_k)=\prod_{j=1}^e\frac{1}{1-B_j}.
\]

{\bf Case 3.}
Now, let (\ref{one single equation}) contain both positive $a_i$ and negative $c_m$.
We shall use the Elliott tricky equality \cite[equation (4)]{E}
\begin{equation}\label{Eli-trick}
\frac{1}{(1-Az^a)(1-Cz^{-c})} = \frac{1}{1-ACz^{a-c}} \left(\frac{1}{1-Az^p}+\frac{1}{1-Cz^{-c}} -1 \right),
\end{equation}
where $a,c\in \mathbb{N}$ and $A,C$ are again monomials in $t_1,\ldots,t_k$. Applying (\ref{Eli-trick}) to a pair of $a_i$ and $c_m$, we shall
replace the product (\ref{products with positive and negative degrees}) by a sum of three similar products with numerators $\pm 1$.
Continuing with the application of (\ref{Eli-trick}) to each of the three expressions in several steps we shall obtain a sum of
products with denominators containing factors which do not depend on $z$ and factors with only positive or only negative degrees of $z$.
Then in order to obtain the expression of $f_0(t_1,\ldots,t_k)$ we handle each summand as in Case 1 and Case 2 of the proof.
In this way we compute the characteristic series of the set of the solutions of the first equation of (\ref{system of homogeneous Diophantine equations}).
Applying the same algorithm on $f_0(t_1,\ldots,t_k)$ we obtain the characteristic series of the solutions of the first two equations of
(\ref{system of homogeneous Diophantine equations}) and continue the process until we find the characteristic series $\chi_S(t_1,\ldots,t_k)$ of the set $S$
of the solutions of the whole system (\ref{system of homogeneous Diophantine equations}).
Below we shall explain why the process stops in a finite number of steps.
\end{proof}

The product (\ref{products with positive and negative degrees}) has $d+(k-e)$ factors depending on $z$.
The original arguments of Elliott are the following. If $a=c$,
then, applying (\ref{Eli-trick}), in any of the three summands the number of factors depending on $z$ is smaller.
If $a>c$, then after applying (\ref{Eli-trick}), one of the summands (the third one) has fewer number of factors depending on $z$.
For the first of the other two factors we replace the factor $(1-Az^a)(1-Cz^{-c})$ with negative degree of $z$ in $1-Cz^{-c}$
by the factor $(1-ACz^{a-c})(1-Az^a)$.
Since $a-c$ is between $-c$ and $a$, this expression is simpler than the original.
For the other factor $(1-ACz^{a-c})(1-Cz^{-c})$ again $a-c$ is between $-c$ and $a$ and the expression is simpler than the original, too.
Similar arguments can be applied for the case $a<c$.


We shall formalize the arguments of Elliott following the Master Thesis \cite{K} of the third named author of this paper.
Given two sequences of nonnegative integers
\[
\alpha=(a_1,\ldots,a_d)\text{ and }\gamma=(c_{e+1},\ldots,c_k),
\]
we denote by $\theta=[\alpha,\gamma]$ the corresponding pair of partitions
\[
[\alpha]=(a_{i_1},\ldots,a_{i_d}),\quad a_{i_1}\geq\cdots\geq a_{i_d},\quad [\gamma]=(c_{m_1},\ldots,c_{m_{k-e}}),\quad c_{m_1}\geq\cdots\geq c_{m_{k-e}}.
\]
Then we define the linear order
\[
\theta=[\alpha,\gamma]\prec[\alpha',\gamma']=\theta',\text{ if }[\alpha]<[\alpha']\text{ or }[\alpha]=[\alpha'],[\gamma]<[\gamma'],
\]
where the order $<$ in $[\alpha]<[\alpha']$ and $[\gamma]<[\gamma']$ is with respect to the usual lexicographic order. Obviously the order $\prec$
satisfies the descending chain condition.

\begin{proposition}\label{from Masther Thesis} The algorithm of Elliott in the proof of Theorem \ref{theorem of Elliott}
stops after a finite number of steps.
\end{proposition}

\begin{proof}
Applying (\ref{Eli-trick}) to the product $\xi=\xi_S(t_1,\ldots,t_k,z)$ from
(\ref{products with positive and negative degrees}) we shall follow the behavior of the two sequences
$\alpha=(a_1,\ldots,a_d)$ and $\gamma=(c_{e+1},\ldots,c_k)$ of the degrees of $z$
and the corresponding pair of partitions $\theta=[\alpha,\gamma]$.
Since the set of all finite integer sequences is well ordered with respect to $\prec$, it is sufficient to show that the statement holds for
the pair $\theta$ if it holds for all pairs $\varphi$ which are smaller with respect to $\prec$, and then to apply inductive arguments.
Without loss of generality we may assume that $a_1\geq\cdots\geq a_d$ and $c_{e+1}\geq\cdots\geq c_k$.
By virtues of (\ref{Eli-trick}) we replace the product (\ref{products with positive and negative degrees})
corresponding to $\theta=[\alpha,\gamma]$
by three products
\[
\xi'=\frac{1}{1-A_1C_1z^{a_1-c_{e+1}}}\prod_{i=1}^d\prod_{j=d+1}^e\prod_{m=e+2}^k\frac{1}{(1-A_iz^{a_i})(1-B_j)(1-C_mz^{-c_m})}
\]
\[
\xi''=\frac{1}{1-A_1C_1z^{a_1-c_{e+1}}}\prod_{i=2}^d\prod_{j=d+1}^e\prod_{m=e+1}^k\frac{1}{(1-A_iz^{a_i})(1-B_j)(1-C_mz^{-c_m})}
\]
\[
\xi'''=\frac{1}{1-A_1C_1z^{a_1-c_{e+1}}}\prod_{i=2}^d\prod_{j=d+1}^e\prod_{m=e+2}^k\frac{1}{(1-A_iz^{a_i})(1-B_j)(1-C_mz^{-c_m})},
\]
corresponding to the pairs $\theta'=[\alpha',\gamma']$, $\theta''=[\alpha'',\gamma'']$, $\theta'''=[\alpha''',\gamma''']$,
respectively.

{\bf Case 1.} Let $a_1=c_1$. Then
\[
\theta'=[(a_1-c_{e+1}=0,a_1,a_2,\ldots,a_d),(c_{e+2},\ldots,c_k)],
\]
\[
\theta''=[(a_2,\ldots,a_d),(c_{e+1},c_{e+2},\ldots,c_k)],
\]
\[
\theta'''=[(a_2,\ldots,a_d),(c_{e+2},\ldots,c_k)].
\]
Since $[c_{e+2},\ldots,c_k]<[c_{e+1},c_{e+2},\ldots,c_k]$ and $[a_2,\ldots,\ldots,a_d]<[a_1,a_2,\ldots,\ldots,a_d]$ we obtain that
$\theta',\theta'',\theta'''\prec\theta$ and we can apply inductive arguments.

{\bf Case 2.} Let $a_1>c_{e+1}$. Then $\theta',\theta'',\theta'''$ are, respectively,
\[
\theta'=[\alpha',\gamma']=[(a_1-c_{e+1},a_1,a_2,\ldots,a_d),(c_{e+2},\ldots,c_k)],
\]
\[
\theta''=[\alpha'',\gamma'']=[(a_1-c_{e+1},a_2,\ldots,a_d),(c_{e+1},c_{e+2},\ldots,c_k)],
\]
\[
\theta'''=[\alpha''',\gamma''']=[(a_1-c_{e+1},a_2,\ldots,a_d),(c_{e+2},\ldots,c_k)].
\]
Since $[\alpha'']=[\alpha''']<[\alpha]$ we have that
$\theta'',\theta'''\prec\theta$ and we can apply inductive arguments for them. But we have that $[\alpha']>[\alpha]$
and $\theta'\succ\theta$. Applying
(\ref{Eli-trick}) to $\theta'$ we obtain three pairs of partitions:
\[
(\theta')'=[(\alpha')',(\gamma')']=[(a_1-c_{e+1},a_1-c_{e+2},a_1,a_2,\ldots,a_d),(c_{e+3},\ldots,c_k)],
\]
\[
(\theta')''=[(\alpha')'',(\gamma')'']=[(a_1-c_{e+1},a_1-c_{e+2},a_2,\ldots,a_d),(c_{e+2},c_{e+3},\ldots,c_k)],
\]
\[
(\theta')'''=[(\alpha')''',(\gamma')''']=[(a_1-c_{e+1},a_1-c_{e+2},a_2,\ldots,a_d),(c_{e+3},\ldots,c_k)].
\]
Let us assume that
\[
[\alpha]=(\underbrace{a_1,\ldots,a_1}_{r\text{ times}},a_{r+1},\ldots,a_d),\quad a_1>a_{r+1}\geq\cdots\geq a_d.
\]
Then
\[
[(\alpha')'']=[(\alpha')''']=[\underbrace{a_1,\ldots,a_1}_{r-1\text{ times}},a_1-c_{e+1},a_1-c_{e+2},a_{r+1},\ldots,a_d]<[\alpha]
\]
and again $(\theta')'',(\theta')'''\prec\theta$. Hence we have a problem with $(\theta')'$ only.

The application of (\ref{Eli-trick}) to $(\theta')'$ gives three pairs $((\theta')')',((\theta')')'',((\theta')')'''$.
The latter two, $((\theta')')''$ and $((\theta')')'''$, are smaller than $\theta$ and
\[
((\theta')')'=[(a_1,\ldots,a_d,a_1-c_{e+1},a_1-c_{e+2},a_1-c_{e+3}),(c_{e+4},\ldots,c_k)]\succ\theta.
\]
Continuing to apply (\ref{Eli-trick}), we obtain in each step pairs of partitions which are smaller than $\theta$, and the pairs
\[
[(a_1,\ldots,a_d,a_1-c_{e+1},\ldots,a_1-c_{e+j}),(c_{e+j+1},\ldots,c_k)]\succ\theta.
\]
Finally, we shall reach the pair
\[
[(a_1,\ldots,a_d,a_1-c_{e+1},\ldots,a_1-c_k),(0)]
\]
corresponding to the product
\[
\prod_{i=1}^d\frac{1}{1-A_iz^{a_i}}\prod_{m=e+1}^k\frac{1}{1-A_1C_mz^{a_1-c_m}}\prod_{j=d+1}^e\frac{1}{1-B_j}
\]
which we can handle as in Case 1 of Theorem \ref{theorem of Elliott}.

{\bf Case 3.} Let $a_1<c_{e+1}$. Applying (\ref{Eli-trick}) to $\theta$ gives pairs of partitions
\[
\theta'=[\alpha',\gamma']=[(a_1,a_2,\ldots,a_d),(c_{e+1}-a_1,c_{e+2},\ldots,c_k)],
\]
\[
\theta''=[\alpha'',\gamma'']=[(a_2,\ldots,a_d),(c_{e+1}-a_1,c_{e+1},c_{e+2},\ldots,c_k)],
\]
\[
\theta'''=[\alpha''',\gamma''']=[(a_2,\ldots,a_d),(c_{e+1}-a_1,c_{e+2},\ldots,c_k)]
\]
with $\theta',\theta'''\prec\theta$. As in Case 2, we replace $\theta''$ by a sequence
\[
[(a_{j+1},\ldots,a_d),(c_{e+1}-a_1,\ldots,c_{e+1}-a_j,c_{e+1},c_{e+2},\ldots,c_k)],\quad j=2,\ldots,d,
\]
until we obtain $[(0),(c_{e+1}-a_1,\ldots,c_{e+1}-a_d,c_{e+1},c_{e+2},\ldots,c_k)]$
and then handle the corresponding product as in Case 2 of Theorem \ref{theorem of Elliott}.
\end{proof}

The following theorem is a modification of Theorem \ref{theorem of Elliott} for systems of Diophantine inequalities.

\begin{theorem}\label{theorem for system of inequalities}
Let $a_{ij},a_i\in\mathbb{Z}$, $i=1,\ldots,m$, $j=1,\ldots,k$,
be arbitrary integers. Then the characteristic series $\chi_S(t_1,\ldots,t_k)$ of the set $S$ of the solutions in nonnegative integers
of the system of Diophantine inequalities
\[
\begin{array}{|cl}
a_{11} x_1 + \cdots + a_{1k} x_k + a_1\geq 0 \\
\cdots&\\
a_{m1} x_1 + \cdots + a_{mk} x_k + a_m\geq 0 \\
\end{array}
\]
is a nice rational function.
\end{theorem}

\begin{proof}
We repeat the arguments from the proof of Theorem \ref{theorem of Elliott} using Lemma \ref{Laurent series}
instead of Lemma \ref{Laurent series for equations}. Consider the linear Diophantine inequality
\[
a_1x_1+\cdots+a_kx_k+a\geq 0,\quad a_i,a\in\mathbb{Z},
\]
and the nice rational function
\[
f(t_1,\ldots,t_k)=t_1^{r_1}\cdots t_k^{r_k}\prod\frac{1}{1-t_1^{q_{i1}}\cdots t_k^{q_{ik}}}
\]
with a set $S$ of solutions in nonnegative integers. By Lemma \ref{Laurent series} we have to compute the component
\[
\chi_S(t_1,\ldots,t_k)=\sum_{n=0}^{\infty}f_n(t_1,\ldots,t_k)
\]
of the Laurent series
\[
\xi_S(t_1,\ldots,t_k,z)=z^af(t_1z^{a_1},\ldots,t_kz^{a_k})=\sum_{n=-\infty}^{\infty}f_n(t_1,\ldots,t_k)z^n.
\]
Applying the algorithm of Elliott to the part $\displaystyle \prod\frac{1}{1-D_iz^{d_i}}$ of
\[
z^af(t_1z^{a_1},\ldots,t_kz^{a_k})=z^bB\prod\frac{1}{1-D_iz^{d_i}},
\]
where $B$ and $D_i$ are monomials in $t_1,\ldots,t_k$,
we present $z^af(t_1z^{a_1},\ldots,t_kz^{a_k})$ as a sum of products of the form
\[
\xi^+=z^dE\prod\frac{1}{1-G_i}\prod_{h_j>0}\frac{1}{1-H_jz^{h_j}}\text{ and }\xi^-=z^dE\prod\frac{1}{1-G_i}\prod_{h_j<0}\frac{1}{1-H_jz^{h_j}},
\]
where again $E,G_i,H_j$ are monomials in $t_1,\ldots,t_k$. In order to complete the proof we have to determine the contribution
of each summand $\xi^+$ and $\xi^-$ to $\chi_S(t_1,\ldots,t_k)$.

{\bf Case 1.} The exponent $d$ of $z$ in $\xi^+$ satisfies $d\geq 0$. Then the whole $\xi^+$ contributes to $\chi_S(t_1,\ldots,t_k)$.

{\bf Case 2.} The exponent $d$ of $z$ in $\xi^-$ satisfies $d\leq 0$. If $d<0$, then $\xi^-$ does not contribute to $\chi_S(t_1,\ldots,t_k)$
because its expansion as a Laurent series contains only negative degrees of $z$. If $d=0$, then the contribution of $\xi^-$ is
$\displaystyle E\prod\frac{1}{1-G_i}$.

{\bf Case 3.} The exponent $d$ of $z$ in $\xi^+$ satisfies $d<0$. We expand $\displaystyle \prod_{h_j>0}\frac{1}{1-H_jz^{h_j}}$ as
\[
\prod_{h_j>0}\frac{1}{1-H_jz^{h_j}}=1+K_1z+K_2z^2+\cdots+K_{d-1}z^{d-1}+z^dL(z),
\]
where $L(z)\in {\mathbb C}[[t_1,\ldots,t_k,z]]$. Then
\[
\xi^+=E\prod\frac{1}{1-G_i}(z^{-d}+K_1z^{-d+1}+K_2z^{-d+2}+\cdots+K_{d-1}z^{-1})
+E\prod\frac{1}{1-G_i}L(z)
\]
and the contribution of $\xi^+$ to $\chi_S(t_1,\ldots,t_k)$ is $\displaystyle E\prod\frac{1}{1-G_i}L(z)$.

{\bf Case 4.} The exponent $d$ of $z$ in $\xi^-$ satisfies $d>0$. As in Case 3 we have
\[
\prod_{h_j<0}\frac{1}{1-H_jz^{h_j}}=1+K_1z^{-1}+K_2z^{-2}+\cdots+K_dz^{-d}+z^{-(d+1)}L(z),
\]
where $L(z)\in {\mathbb C}[[t_1,\ldots,t_k,z^{-1}]]$. Then
\[
\xi^-=E\prod\frac{1}{1-G_i}(z^d+K_1z^{d-1}+K_2z^{d-2}+\cdots+K_{d-1}z+K_d)
+z^{-1}E\prod\frac{1}{1-G_i}L(z)
\]
and the contribution of $\xi^-$ is
\[
E\prod\frac{1}{1-G_i}(z^d+K_1z^{d-1}+K_2z^{d-2}+\cdots+K_{d-1}z+K_d).
\]
\end{proof}

There are several algorithms for solving linear systems of Diophantine equations and inequalities
using the method of Elliott,
see e.g. Domenjoud, Tom\'as \cite{DT}, Pasechnik \cite{Pa} and Xin \cite{X}.

Applying the result of Elliott we start with a nice rational function and obtain the result also in the form
of a nice rational function. See Stanley \cite{S1} for further discussions and applications of the approach of Elliott.

The next theorem of Blakley \cite{Bla} gives another point of view of the problem.

\begin{theorem}\label{theorem of Blakley}
Let
\[
f(t_1,\ldots,t_k)=\prod_{i=1}^n\frac{1}{1-t_1^{a_{i1}}\cdots t_k^{a_{ik}}}=\sum_{b_j\geq 0}\beta(b_1,\ldots,b_k)t_1^{b_1}\cdots t_k^{b_k},
\quad \beta(b_1,\ldots,b_k)\in{\mathbb N}.
\]
Then there is a finite decomposition of ${\mathbb N}^k$ such that the coefficients $\beta(b_1,\ldots,b_k)$ are polynomials
of degree $n-k$ in $b_1,\ldots,b_k$ on each piece.
\end{theorem}

In the notation of Theorem \ref{theorem of Blakley} Sturmfels \cite{St} proposed a method to find such a decomposition and
the polynomials which express the coefficients $\beta(b_1,\ldots,b_k)$.

\section{The algorithm of Xin}
In this section we give an idea for the algorithm of Xin \cite{X} in a form suitable for our purposes.
The algorithm is based on two easy observations.

\begin{lemma}\label{decomposition in irreducibles}
Let
\[
q(t_1,\ldots,t_k,z)=1-t_1^{a_1}\cdots t_k^{a_k}z^b,\quad a_i\in{\mathbb N},b\in{\mathbb Z}.
\]

{\rm (i)} If $b>0$, then $q(t_1,\ldots,t_k,z)$ decomposes as a product of irreducible polynomials in ${\mathbb Q}[t_1,\ldots,t_k,z]$
with constant terms (as polynomials in $z$) equal to $1$.

{\rm (ii)} If $b<0$, then $q(t_1,\ldots,t_k,z)$ is decomposed as
\[
q(t_1,\ldots,t_k,z)=z^b\prod_{i=1}^mu_i(t_1,\ldots,t_k,z),
\]
where the irreducible polynomials $u_i(t_1,\ldots,t_k,z)\in{\mathbb Q}[t_1,\ldots,t_k,z]$
are with leading terms (as polynomials in $z$) equal to $z^{n_i}$, $n_i\geq 1$.
\end{lemma}

\begin{proof} (i) Let $b>0$ and let $q(t_1,\ldots,t_k,z)=1-t_1^{a_1}\cdots t_k^{a_k}z^b$ decompose as
\[
q(t_1,\ldots,t_k,z)=\prod_{i=1}^mu_i(t_1,\ldots,t_k,z), \quad u_i(t_1,\ldots,t_k,z)\in{\mathbb Z}[t_1,\ldots,t_k,z].
\]
Comparing the constant term 1 of $q(t_1,\ldots,t_k,z)$ with respect to $z$ with the product of the constant terms of the factors
$u_i(t_1,\ldots,t_k,z)$ we derive that the constant terms of $u_i(t_1,\ldots,t_k,z)$ belong to $\mathbb Q$,
i.e. we may assume that they are equal to 1.

(ii) If $b<0$ we present $q(t_1,\ldots,t_k,z)$ in the form
\[
q(t_1,\ldots,t_k,z)=\frac{1}{z^c}q_1(t_1,\ldots,t_k,z), \quad q_1(t_1,\ldots,t_k,z)=z^c-t_1^{a_1}\cdots t_k^{a_k},\quad c=-b.
\]
As in (i), comparing the leading monomials $z^c$ of $q_1(t_1,\ldots,t_k,z)$ and the product of the leading monomials
$u_i(t_1,\ldots,t_k,z)$ we derive the proof of (ii).
\end{proof}

\begin{proposition}\label{partial fraction}
Let
\begin{equation}\label{equation for partial fraction}
\begin{split}
f(t_1,\ldots,t_k,z)&=g(t_1,\ldots,t_k,z,z^{-1})\prod_{i=1}^m\frac{1}{1-t_1^{a_{i1}}\cdots t_k^{a_{ik}}z^{b_i}}\\
&=\sum_{n=-\infty}^{\infty}f_n(t_1,\ldots,t_k)z^n,\\
\end{split}
\end{equation}
where $a_{ij}\in{\mathbb N}$, $b_i\in{\mathbb Z}$,
$g(t_1,\ldots,t_k,z,z^{-1})\in{\mathbb Z}[t_1,\ldots,t_k,z,z^{-1}]$
is a polynomial in $t_1,\ldots,t_k$ and a Laurent polynomial in $z$,
and $f_n(t_1,\ldots,t_k)\in \mathbb{Q}[[t_1,\ldots,t_k]]$.
Let the partial fraction decomposition of $f(t_1,\ldots,t_k,z)$ be
\[
f(t_1,\ldots,t_k,z)=p(t_1,\ldots,t_k,z)+\sum_{l,d}\frac{p_{ld}(t_1,\ldots,t_k,z)}{q_l^d(t_1,\ldots,t_k,z)}
+\sum_{j,e}\frac{r_{je}(t_1,\ldots,t_k,z)}{s_j^e(t_1,\ldots,t_k,z)},
\]
where $p,p_{ld},q_l,r_{je},s_j\in {\mathbb Q}[t_1,\ldots,t_k,z]$,
$q_l$ and $s_j$ are irreducible in ${\mathbb Q}[t_1,\ldots,t_k,z]$, 
$\deg_zp_{ld}<\deg_zq_l$, $\deg_zr_{je}<\deg_zs_j$, the constant term $q_l(t_1,\ldots,t_k,0)$ of each $q_l$ be nonzero and belong to $\mathbb Q$ and
the constant term $s_j(t_1,\ldots,t_k,0)$ of each $s_j$ be a polynomial of positive degree in ${\mathbb Q}[t_1,\ldots,t_k]$ (or $s_j(t_1,\ldots,t_k,0)=z$).
Then
\[
h(t_1,\ldots,t_k,z)=\sum_{n=0}^{\infty}f_n(t_1,\ldots,t_k)z^n
=p(t_1,\ldots,t_k,z)+\sum_{l,d}\frac{p_{ld}(t_1,\ldots,t_k,z)}{q_l^d(t_1,\ldots,t_k,z)}
\]
and $f_0(t_1,\ldots,t_k)=h(t_1,\ldots,t_k,0)$.
\end{proposition}

\begin{proof}
The polynomials $q^d(t_1,\ldots,t_k,z)$ and $s^e(t_1,\ldots,t_k,z)$ in the denominators
in the expression of $f(t_1,\ldots,t_k,z)$ are of the form prescribed in
Lemma \ref{decomposition in irreducibles} (or $s(t_1,\ldots,t_k,z)=z$). Hence we may assume that
\[
q(t_1,\ldots,t_k,z)=1+zv(t_1,\ldots,t_k,z),
\]
\[
s(t_1,\ldots,t_k,z)=z^n+w(t_1,\ldots,t_k,z),
\]
$v, w\in{\mathbb Q}[t_1,\ldots,t_k,z]$, $\deg_zw<n$,
$\deg w(t_1,\ldots,t_k,0)>0$ (or $s(t_1,\ldots,t_k,z)=z$).
The expansion of the fractions with denominators of the form $q^d(t_1,\ldots,t_k,0)$ belongs to ${\mathbb Q}[[t_1,\ldots,t_k,z]]$
because the expression
\[
\frac{1}{q^d(t_1,\ldots,t_k,z)}=\frac{1}{(1+zv(t_1,\ldots,t_k,z))^d}
\]
\[
=(1+zv(t_1,\ldots,t_k,z)+z^2v^2(t_1,\ldots,t_k,z)+\cdots)^d
\]
does not involve negative degrees of $z$. By similar arguments, when $\deg s(t_1,\ldots,t_k,0)>0$, all monomials in the expansion of
\[
\frac{1}{s^e(t_1,\ldots,t_k,z)}=\frac{1}{(z^n+w(t_1,\ldots,t_k,z))^e}=\frac{1}{z^{nd}(1+u(t_1,\ldots,t_k,z^{-1}))^e}
\]
\[
=\frac{1}{z^{nd}}(1+u(t_1,\ldots,t_k,z^{-1})+u^2(t_1,\ldots,t_k,z^{-1})+\cdots)^d
\]
involve factors $z^{-m}$ with negative degrees of $z$ with $m\geq n$. Since $m$ is larger than the degree in $z$
of the corresponding numerator $r(t_1,\ldots,t_k,z)$,
the expansion of these fractions contains only negative degrees of $z$ and does not contribute to
$h(t_1,\ldots,t_k,z)$. When $s(t_1,\ldots,t_k,z)=z$ the numerator $r(t_1,\ldots,t_k,z)$ does not depend on $z$
and hence these fractions do not participate in $h(t_1,\ldots,t_k,z)$ again.
\end{proof}

\begin{algorithm}\label{algorithm of Xin}
We want to solve the homogeneous linear Diophantine system of equations and inequalities
\begin{equation}\label{general system to solve by Xin}
\begin{array}{|rcrl}
a_{11} x_1& + \cdots +& a_{1k} x_k &= 0 \\
&\cdots&&\\
a_{l1} x_1& + \cdots + &a_{lk} x_k &= 0 \\
a_{l+1,1} x_1& + \cdots +& a_{l+1,k} x_k&\geq 0 \\
&\cdots&&\\
a_{m1} x_1& + \cdots +& a_{mk} x_k&\geq 0. \\
\end{array}
\end{equation}
We start with the function
\[
u(t_1,\ldots,t_k)=\prod_{i=1}^k\frac{1}{1-t_i},
\]
replace the variables $t_i$ by $t_iz^{a_{1i}}$, $i=1,\ldots,k$, and expand $u(t_1z^{a_{11}},\ldots,t_kz^{a_{1k}})$ in the form (\ref{equation for partial fraction})
\[
f(t_1,\ldots,t_k,z)=\sum_{n=-\infty}^{\infty}f_n(t_1,\ldots,t_k)z^n.
\]
Applying Proposition \ref{partial fraction} we obtain
\[
h(t_1,\ldots,t_k,z)=\sum_{n=0}^{\infty}f_n(t_1,\ldots,t_k)z^n
=p(t_1,\ldots,t_k,z)+\sum_{l,d}\frac{p_{ld}(t_1,\ldots,t_k,z)}{q_l^d(t_1,\ldots,t_k,z)}.
\]
All polynomials $q_l(t_1,\ldots,t_k,z)$ in the denominators are divisors of some
$1-t_1^{b_1}\cdots t_k^{b_k}z^c$, multiplying the numerators and denominators with suitable polynomials we present
$h(t_1,\ldots,t_k,z)$ as a fraction with denominator in the form $\prod(1-t_1^{b_1}\cdots t_k^{b_k}z^c)$, i.e.
the result is a nice rational function. If we start with an equation
$a_{11} x_1 + \cdots + a_{1k} x_k = 0$ from (\ref{general system to solve by Xin})
we take $f_0(t_1,\ldots,t_k)=h(t_1,\ldots,t_k,0)$, continue the work with
$f_0(t_1,\ldots,t_k)$ and handle the next equation or inequality of (\ref{general system to solve by Xin}).
If we have an inequality $a_{11} x_1 + \cdots + a_{1k} x_k \geq 0$, we make the next step with the function
$h(t_1,\ldots,t_k,1)$ which takes into account all $f_n(t_1,\ldots,t_k)$, $n\geq 0$.
Continuing in the same way, we obtain in each step a nice rational function which is the characteristic series
of the solutions of the first several equations and inequalities of (\ref{general system to solve by Xin}).
At the final step, we obtain the characteristic series of the solutions of the whole system.
\end{algorithm}

\begin{example}\label{example for algorithm of Xin}
We start with the system from Example \ref{solution of system by Gordan}
\[
\begin{array}{|rrrrl}
x_1&+2x_2&-x_3&-x_4&=0\\
2x_1&+3x_2&-2x_3&-x_4&=0.\\
\end{array}
\]
By Algorithm (\ref{algorithm of Xin}),
\[
u(t_1,t_2,t_3,t_4)=\frac{1}{(1-t_1)(1-t_2)(1-t_3)(1-t_4)},
\]
\[
f(t_1,t_2,t_3,t_4,z)=u(t_1z,t_2z^2,t_3z^{-1},t_4z^{-1})
\]
\[
=\frac{1}{(1-t_1)(1-t_2z^2)(1-t_3z^{-1})(1-t_4z^{-1})}
\]
\[
=\frac{t_3^2}{(t_3-t_4)(1-t_1t_3)(1-t_2t_3^2)(z-t_3)}-\frac{t_4^2}{(t_3-t_4)(1-t_1t_4)(1-t_2t_4^2)(z-t_4)}
\]
\[
+\frac{(1+t_1t_3+t_1t_4+t_2t_3t_4+(t_1+t_2t_3+ t_2t_4+ t_1t_2t_3t_4)z)t_2}
{(t_2-t_1^2)(1-t_2t_3^2)(1-t_2t_4^2)(1-t_2z^2)}
\]
\[
-\frac{t_1^2}{(t_2-t_1^2)(1-t_1t_3)(1-t_1t_4)(1-t_1z)},
\]
\[
h(t_1,t_2,t_3,z)=\frac{(1+t_1t_3+t_1t_4+t_2t_3t_4+(t_1+t_2t_3+ t_2t_4+ t_1t_2t_3t_4)z)t_2}
{(t_2-t_1^2)(1-t_2t_3^2)(1-t_2t_4^2)(1-t_2z^2)}
\]
\[
-\frac{t_1^2}{(t_2-t_1^2)(1-t_1t_3)(1-t_1t_4)(1-t_1z))},
\]
\[
f_0(t_1,t_2,t_3,t_4)=h(t_1,t_2,t_3,0)=\frac{1+t_2t_3t_4-t_1t_2t_3^2t_4-t_1t_2t_3t_4^2}{(1-t_1t_3)(1-t_1t_4)(1-t_2t_3^2)(1-t_2t_4^2)}.
\]
We continue in the same way with the second equation and present $f_0(t_1z^2,t_2z^3,t_3z^{-2},t_4z^{-1})$ as a sum of partial fractions.
Finally we obtain the characteristic series of the solutions of
(\ref{example homogeneous system})
\[
\chi_S(t_1,t_2,t_3,t_4)=\frac{1}{(1-t_1t_3)(1-t_2t_3t_4)}=\sum_{m,n\geq 0}(t_1t_3)^m(t_2t_3t_4)^n.
\]
This means that all solutions of the system are
\[
m(1,0,1,0)+n(0,1,1,1),\quad m,n\in{\mathbb N},
\]
i.e. the minimal solutions are $(1,0,1,0)$ and $(0,1,1,1)$.
\end{example}

\section{Our solution of the problem of Robles-P\'erez and Rosales}

We shall illustrate the method of Elliott and the algorithm of Xin on the example of the Diophantine transport problem given in \cite{RPR}
which was one of the two main motivations of the present project. As stated in \cite{RPR}, the example is the following.

{\it A transport company carries cars from the factory to a dealer using small and large trucks with a capacity of three and six cars.
The trucks cost for the company $1200$ and $1500$ euros, respectively.
The company receives from the dealer $300$ euros for each transported car and offers as a bonus the transportation of an additional car without charge.
The company considers that the ordered transport is profitable when it has a profit of at least $900$ euros.
How many cars must be transported at least in order to achieve that purpose?}

If $y$ denotes the required number of cars, $x_3$ and $x_6$ are the numbers of the small and the large trucks, respectively,
the problem is equivalent to the linear Diophantine system
\[
\begin{array}{|ll}
300y &\geq 1200x_3+1500x_6+900\\
y+1 &\leq  3x_3+6x_6\\
\end{array}
\]
which after a simplification has the form:
\begin{equation}\label{systemD}
\begin{array}{|ccl}
y &\geq &4x_3+5x_6+3\\
y &\leq &3x_3+6x_6-1\\
\end{array} \qquad \Longrightarrow \qquad
\begin{array}{|rcl}
-4x_3-5x_6+y-3&\ge &0\\
3x_3+6x_6-y-1 &\ge &0.\\
\end{array}
\end{equation}

The goal of the paper \cite{RPR} was to prove that the set $T$ of the integers $n$
for which the system (\ref{systemD}) has a solution $(x_3,x_6,y)=(r_3,r_6,n)\in{\mathbb N}^3$
together with 0 forms a submonoid of $({\mathbb N},+)$, and to give algorithmic procedures how to compute $T$.
We shall extend this goal and show how to find the set $S$ of all solutions $(r_3,r_6,n)\in{\mathbb N}^3$ of the system.
In particular, we shall discuss the relations between the profit and the solutions of the system.

\begin{remark}
As in the one-dimensional case considered in \cite{RPR}, it is easy to see that the set $S$ of solutions
$(r_3,r_6,n)$ of the system (\ref{systemD}), together with $(0,0,0)$ forms a submonoid of $({\mathbb N}^3,+)$.
\end{remark}

\begin{solution of system}
Applying Theorem \ref{reduction from arbitrary to homogeneous} and Algorithm \ref{algorithm of Xin},
the first inequality of (\ref{systemD}) is replaced by the equation
$-4x_3-5x_6+y-3t=0$ which has to be solved for $t=1$. We start with the function
\[
u(x_3,x_6,y,t)=\frac{1}{(1-x_3)(1-x_6)(1-y)(1-t)}
\]
and replace the variables $x_3,x_6,y,t$ by $x_3z^{-4},x_6z^{-5},yz,tz^{-1}$, respectively.
Presenting the obtained function $f(x_3,x_6,y,t,z)$ as a sum of partial fractions with respect to $z$
\[
f(x_3,x_6,y,t,z)=u(x_3z^{-4},x_6z^{-5},yz,tz^{-3})
=\sum_{n=-\infty}^{\infty}f_n(x_3,x_6,y,t)z^n,
\]
we obtain
\[
h(x_3,x_6,y,t,z)=\sum_{n=0}^{\infty}f_n(x_3,x_6,y,t)z^n=\frac{1}{(1-x_3y^4)(1-x_6y^5)(1-y^3t)(1-yz)}.
\]
By Theorem \ref{reduction from arbitrary to homogeneous} the solutions of the first inequality in
(\ref{systemD}) are obtained from the solutions $(x_3,x_6,y,t)$ with $t=1$.
In the expansion
\[
h(x_3,x_6,y,t,z)=\sum_{d=0}^{\infty}h_d(x_3,x_6,y,z)t^d
\]
these solutions correspond to the coefficient $h_1(x_3,x_6,y,z)$. Hence
\[
h_1(x_3,x_6,y,z)=\frac{y^3}{(1-x_3y^4)(1-x_6y^5)(1-yz)}.
\]
Instead of the second inequality of (\ref{systemD}) we consider the equation
\[
3x_3+6x_6-y-v=0.
\]
We start with the function
\[
p(x_3,x_6,y,v)=\frac{1}{1-v}h_1(x_3,x_6,y,1)=\frac{y^3}{(1-x_3y^4)(1-x_6y^5)(1-y)(1-v)}
\]
and applying Algorithm \ref{algorithm of Xin} we obtain
\[
q(x_3,x_6,y,v,w)=p(x_3w^3,x_6w^6,yw^{-1},vw^{-1})
\]
\[
=\frac{y^3}{w^3(1-x_3y^4w^{-1})(1-x_6y^5w)(1-yw^{-1})(1-vw^{-1})}=\sum_{m=-\infty}^{\infty}q_m(x_3,x_6,y,v)w^m,
\]
\[
r(x_3,x_6,y,v,w)=\sum_{m=0}^{\infty}q_m(x_3,x_6,y,v)w^m
\]
\[
=\frac{x_6^3y^{18}}{(1-x_6y^6)(1-x_3x_6y^9)(1-x_6y^5v)(1-x_6y^5w)}=\sum_{k=0}^{\infty}r_k(x_3,x_6,y,w)v^k,
\]
\[
r_1(x_3,x_6,y,w)=\frac{x_6^4y^{23}}{(1-x_6y^5w)(1-x_6y^6)(1-x_3x_6y^9)},
\]
\[
\chi_S(x_3,x_6,y)=r_1(x_3,x_6,y,1)=\frac{x_6^4y^{23}}{(1-x_6y^5)(1-x_6y^6)(1-x_3x_6y^9)}.
\]
Hence the set $S$ of all solutions $(r_3,r_6,n)$ of the system (\ref{systemD}) are
\[
(r_3,r_6,n)=(0,4,23)+c_1(0,1,5)+c_2(0,1,6)+c_3(1,1,9),\quad c_1,c_2,c_3\in{\mathbb N}.
\]
Replacing $x_3$ and $x_6$ with 1 in $\chi_S(x_3,x_6,y)$ we obtain the characteristic series of the set $T$ of the possible values of $n$:
\[
\chi_T(y)=\chi_S(1,1,y)=\frac{y^{23}}{(1-y^5)(1-y^6)(1-y^9)}
\]
which expresses the original solution of the problem in \cite{RPR} in the form of its characteristic series.
\end{solution of system}

\begin{remark}
Using the expression for $\chi_S(x,y,z)$ found above, it is easy to find a relation between the solutions and the corresponding profit.
Since for each transported car the firm gains 300 euros and pays, respectively, 1200 euros and 1500 euros for each small and large truck,
we shall consider the function
\[
\omega(x_3,x_6,y,t)=\chi_S(x_3t^{-4},x_6t^{-5},yt)=\frac{x_6^4y^{23}t^3}{(1-x_6y^5)(1-x_3x_6y^9)(1-x_6y^6t)}
\]
\[
=\sum_{k=3}^{\infty}\omega_k(x_3,x_6,y)t^k=\sum_{k=3}^{\infty}\frac{x_6^{k+1}y^{6k+5}t^k}{(1-x_6y^5)(1-x_3x_6y^9)}.
\]
The firm will have a profit $300k$ euros for all $(r_3,r_6,n)$ such that
$x_3^{r_3}x_6^{r_6}y^n$ participates with a nonzero coefficient $\alpha(r_3,r_6,n)$ in the expansion
of $\omega_k(x_3^{r_3}x_6^{r_6}y^n)$ as a power series:
\[
\omega_k(x_3,x_6,y)=\frac{x_6^{k+1}y^{6k+5}}{(1-x_6y^5)(1-x_3x_6y^9)}=\sum_{r_3,r_6,n\geq 0}\alpha(r_3,r_6,n)x_3^{r_3}x_6^{r_6}y^n.
\]
In particular,
it is easy to see that the minimal number of transported cars to gain a profit $300k$ is
$n=6k+5$ (plus one car as a bonus) and for this purpose the firm has to use $k+1$ large trucks.
\end{remark}

\section*{Acknowledgements}
The second named author is very grateful to Andreas Weiermann for
his comments about the relations of the problems considered in the paper
with recursion theory and about the results of Seidenberg \cite{Sa} and Moreno-Soc\'{\i}as \cite{MS1, MS2}
in Theorems \ref{result of Moreno-Socias} and \ref{Ackermann function}.


\begin{thebibliography}{99}

\bibitem{AbM}
H. Abdulrab, M. Maksimenko, General solution of linear Diophantine equations and inequations,
in ``Rewriting techniques and applications. 6th International Conference, RTA-95, Kaiserslautern, Germany, April 5-7, 1995'',
Lecture Notes in Computer Science, {\bf 914}, Springer-Verlag, Berlin, 1995, 339-351.

\bibitem{Ack}
W. Ackermann,
Zum Hilbertschen Aufbau der reellen Zahlen,
{\it Math. Ann.}, {\bf 99} (1928), 118-133.

\bibitem{AC}
F. Ajili, E. Contejean,
Complete solving of linear Diophantine equations and inequations without adding variables,
in ``Principles and Practice of Constraint Programming -- CP'95 (Cassis, 1995),
Lecture Notes in Comput. Sci., {\bf 976}, Springer-Verlag, Berlin, 1995, 1-17.

\bibitem{A}
G.\ E. Andrews,
MacMahon's partition analysis. I:
The lecture Hall partition theorem, in
B.E. Sagan (ed.) et al., Mathematical Essays in Honor of Gian-Carlo Rota's 65th Birthday,
Prog. Math., {\bf 161}, Birkh\"auser, Boston, MA, 1998, 1-22.

\bibitem{AP}
G.\ E. Andrews, P. Paule,
MacMahon's partition analysis. XII: Plane partitions,
{\it J. Lond. Math. Soc., II. Ser.}, {\bf 76}, (2007), No. 3, 647-666.

\bibitem{APR1}
G.\ E. Andrews, P. Paule, A. Riese,
MacMahon's partition analysis: The Omega package,
{\it Eur. J. Comb.}, {\bf 22} (2001), No. 7, 887-904.

\bibitem{APR2}
G.\ E. Andrews, P. Paule, A. Riese,
MacMahon's partition analysis. VI: A new reduction algorithm,
{\it Ann. Comb.}, {\bf 5} (2001), 251-270.

\bibitem{APo}
M. Aschenbrenner, W.\ Y. Pong,
Orderings of monomial ideals,
{\it Fundam. Math.}, {\bf 181} (2004), No. 1, 27-74.

\bibitem{BX}
L. Bedratyuk, G. Xin,
MacMahon partition analysis and the Poincar\'e series of the algebras of invariants
of ternary and quaternary forms,
{\it Linear Multilinear Algebra}, {\bf 59} (2011), 789-799.

\bibitem{B2}
A. Berele,
Applications of Belov's theorem to the cocharacter
sequence of p.i. algebras,
{\it J. Algebra}, {\bf 298} (2006), 208-214.

\bibitem{B3}
A. Berele,
Properties of hook Schur functions with applications to p. i. algebras,
{\it Adv. Appl. Math.}, {\bf 41} (2008), 52-75.

\bibitem{BBDGK}
F. Benanti, S.Boumova, V. Drensky, G. Genov, P. Koev,
Computing with rational symmetric functions and applications to invariant theory and PI-algebras,
{\it Serdica Math. J.}, {\bf 38} (2012), 137-188.

\bibitem{Bla}
G.\ R. Blakley,
Combinatorial remarks on partitions of a multipartite number,
{\it Duke Math. J.}, {\bf 31} (1964), 335-340; Errata Ibid. {\bf 31} (1964), 718.

\bibitem{BKo}
Y.\ A. Bogoyavlenskiy, D.\ G. Korzun,
A software system for remote solving homogenous linear diophantine equations in non-negative integers (Russian),
{\it St. Petersburg State Polytechnical University J., Computer Science, Telecommunication and Control Systems},
(2010), 1(93),90-99

\bibitem{C}
E. Contejean,
Solving linear Diophantine constraints incrementally, in ``Logic programming (Budapest, 1993)'',
MIT Press Ser. Logic Program., MIT Press, Cambridge, MA, 1993, 532-549.

\bibitem{Di}
L.\ E. Dickson,
Finiteness of the odd perfect and primitive abundant numbers with $n$ distinct prime factors,
{\it Amer. J. Math.}, {\bf 35} (1913), No. 4, 413-422.

\bibitem{DT}
E. Domenjoud, A.\ P. Tom\'as,
From Elliott-MacMahon to an algorithm for general linear constraints on naturals,
in ``Principles and Practice of Constraint Programming -- CP'95 (Cassis, 1995),
Lecture Notes in Comput. Sci., {\bf 976}, Springer-Verlag, Berlin, 1995, 18-35.

\bibitem{DG}
V. Drensky, C.\ K. Gupta,
Constants of Weitzenb\"ock derivations and invariants of unipotent
transformations acting on relatively free algebras,
{\it J. Algebra}, {\bf 292} (2005), 393-428.

\bibitem{E}
E.] B. Elliott,
On linear homogeneous diophantine equations,
{\it Quart. J. Pure Appl. Math.}, {\bf 34} (1903), 348-377.

\bibitem{Eu1}
L. Euler,
Introductio in Analysin Infinitorum, Vol. I,
Bousquet, Lausanne, M.-M. 1748.
Reprinted as: Leonhardi Euleri Opera Omnia, Ser. I, Vol. VIII
(A. Krazer and F. Rudio, eds.), Teubner, Leipzig, 1922.
German translation by H. Maser:
Einleitung in die Analysis des Unendlichen, Erster Teil, Springer-Verlag, Berlin, 1885. Reprinted: 1983.

\bibitem{Eu2}
L. Euler,
Vollstaendige Anleitung zur Algebra. 2. Theil, Von Aufl\"osung algebraischer
Gleichungen und der unbestimmten Analytic,
Kays. Acad. der Wissenschaften, St. Petersburg, 1770.
Reprinted in: Leonhardi Euleri Opera Omnia, Ser. I, Vol. I
(Vollst\"andige Anleitung zur Algebra, mit den Zus\"atzen von Joseph Louis Lagrange, (H. Weber, ed.),
Teubner, Leipzig, 1911,  209-498.
English translation (with 1. Theil):
Elements of Algebra, J. Johnson, London, 1797. Fifth edition (1840) reprinted: Springer-Verlag, New York, 1984.

\bibitem{Eu3}
L. Euler,
De partitione numerorum in partes tam numero quam specie dates,
{\it Noui Commentarii Academiae Scientiarum Imperialis Petropolitanae}, {\bf 14} (1769): I (1770) 168-187.
Reprinted in: Leonhardi Euleri Opera Omnia, Ser. I, Vol. III
(Commentationes Arithmeticae, Vol. II (F. Rudio, ed.), Teubner, Leipzig, 1917, 132-147.

\bibitem{FT}
M. Filgueiras, A.P. Tomas,
A fast method for finding the basis of non-negative solutions to a linear diophantine equation,
{\it J. Symb. Comput.}, {\bf 19} (1995), No. 6, 507-526.

\bibitem{G1}
P. Gordan,
Ueber die Aufl\"osung linearer Gleichungen mit reellen Coefficienten,
{\it Math. Ann.}, {\bf 6} (1873), No. 1, 23-28.

\bibitem{G2}
P. Gordan,
Neuer Beweis des Hilbert'schen Satzes \"uber homogene Functionen,
{\it G\"ott. Nachr.}, 1899, 240-242.

\bibitem{GY}
J.\ H. Grace, A. Young,
The Algebra of Invariants,
Cambridge, Cambridge University Press, 1903. Reprinted: 2010.

\bibitem{H}
D. Hilbert,
\"Uber die Theorie der algebraischen Formen,
{\it Math. Ann.}, {\bf 36} (1890), 473-534.
Reprinted in Gesammelte Abhandlungen. Bd. 2. Algebra. Invariantentheorie. Geometrie,
Julius Springer, Berlin, viii, 1933, 199-257, reprinted Chelsea, New York, 1965.

\bibitem{Ko1}
D.\ Zh. Korzun,
On the existence of generating contex-free grammars for an arbitrary linear Diophantine system (Russian),
{\it Tr. Petrozavodsk. Gos. Univ. Ser. Mat.}, 1999, No. 6, 34-40.

\bibitem{Ko2}
D.\ Zh. Korzun,
A certain mapping between formal grammars and systems of linear Diophantine equations (Russian),
{\it Vestnik Molodykh Uchenykh, Seria Prikladnaya matematika i mekhanika,
Saint-Petersburg Research Center of the Russian Academy of Science}, (2000), No. 3, 50-56.

\bibitem{K}
B.\ S. Kostadinov,
Application of Rational Generating Functions to Algebras with Polynomial Identities (Bulgarian),
M.Sci. Thesis,
Faculty of Mathematics and Informatics,
Sofia University `` St. Kliment Ohridski''
Program ``Mathematics and Mathematical Physics'', 2011.

\bibitem{Kr}
S.\ L. Kryvyi,
Combinatorial method for solving systems of linear constraints (Russian),
{\it Kibern. Sist. Anal.}, {\bf 50} (2014), No. 4, 14-26.
Translation: {\it Cybern. Syst. Anal.}, {\bf 50} (2014), No. 4, 495-506.

\bibitem{MM}
P.\ A. MacMahon,
Combinatory Analysis, vols. 1 and 2,
Cambridge Univ. Press. 1915, 1916.
Reprinted in one volume: Chelsea, New York, 1960.

\bibitem{MS1}
G. Moreno Soc\'{\i}as,
An Ackermannian polynomial ideal,
Applied Algebra, Algebraic Algorithms and Error-Correcting Codes,
Proc. 9th Int. Symp., AAECC-9, New Orleans/LA (USA) 1991,
Lect. Notes Comput. Sci., {\bf 539}, Springer-Verlag, Berlin, 1991, 269-280.


\bibitem{MS2}
G. Moreno Soc\'{\i}as,
Length of polynomial ascending chains and primitive recursiveness,
{\it Math. Scand.},  {\bf 71}  (1992), No. 2, 181-205.

\bibitem{Pa}
D.\ V. Pasechnik,
On computing Hilbert bases via the Elliot-MacMahon algorithm,
{\it Theor. Comput. Sci.}, {\bf 263} (2001), No. 1-2, 37-46.

\bibitem{PCVT}
P. Pis\'on-Casares, A. Vigneron-Tenorio,
${\mathbb N}$-solutions to linear systems over ${\mathbb Z}$,
{\it Linear Algebra Appl.}, {\bf 384} (2004), 135-154.

\bibitem{RPR}
A.\ M. Robles-P\'{e}rez, J. C. Rosales,
On a transport problem and monoids of non-negative integers,
arXiv:1611.02627v2 [math.GR],
{\it Aequat. Math.}, {\bf 92} (2018), 661-670.

\bibitem{RPR1}
A.\ M. Robles-P\'{e}rez, J. C. Rosales,
Numerical semigroups in a problem about cost-effective transport,
{\it Forum Math.}, {\bf 29} (2017), No. 2, 329-345.

\bibitem{RGGB}
J.\ C. Rosales,  P.\ A. Garc\'{\i}a-S\'anchez, J.\ I. Garc\'{\i}a-Garc\'{\i}a,  M.\ B. Branco,
Systems of inequalities and numerical semigroups,
{\it J. Lond. Math. Soc., II. Ser.}, {\bf 65} (2002), No. 3, 611-623.

\bibitem{Sa}
A. Seidenberg,
On the length of a Hilbert ascending chain,
{\it Proc. Amer. Math. Soc.}, {\bf 29} (1971), 443-450.

\bibitem{Sch}
A. Schrijver,
Theory of Linear and Integer Programming,
Wiley-Interscience Series in Discrete Mathematics,
A Wiley-Interscience Publication, John Wiley \& Sons Ltd., Chichester, 1986.
Reprinted 1998.

\bibitem{S1}
R. Stanley,
Linear homogeneous Diophantine equations and magic labelings of graphs,
{\it Duke Math. J.}, {\bf 40} (1973), 607-632.

\bibitem{St}
B. Sturmfels,
On vector partition functions,
{\it J. Comb. Theory, Ser. A}, {\bf 72} (1995), No. 2, 302-309.

\bibitem{T}
A.\ P. Tom\'as,
On Solving Linear Diophantine Constraints,
Ph.D. Thesis, Universidade do Porto, 1997.

\bibitem{TF1}
A.\ P. Tomas, M. Filgueiras,
Solving linear Diophantine equations using the geometric structure of the solution space,
in ``Rewriting techniques and applications (Sitges, 1997)'',
Lecture Notes in Comput. Sci., {\bf 1232}, Springer-Verlag, Berlin, 1997, 269-283.

\bibitem{TF2}
A.\ P. Tomas, M. Filgueiras,
An algorithm for solving systems of linear Diophantine equations in naturals,
in ``Progress in artificial intelligence (Coimbra, 1997)'',
Lecture Notes in Comput. Sci., {\bf 1323}, Lecture Notes in Artificial Intelligence,
Springer-Verlag, Berlin, 1997, 73-84.

\bibitem{X}
G. Xin,
A fast algorithm for MacMahon's partition analysis,
{\it Electron. J. Comb.}, {\bf 11} (2004), No. 1, Research paper R58.

\end{thebibliography}
\end{document}